\documentclass{amsart}
\usepackage{amsmath, amssymb, amsthm}

\theoremstyle{plain}
\newtheorem{theorem}{Theorem}[section]
\newtheorem{lemma}[theorem]{Lemma}
\newtheorem{corollary}[theorem]{Corollary}

\theoremstyle{definition}
\newtheorem{definition}[theorem]{Definition}

\newcommand{\N}{\mathbb{N}}
\newcommand{\Z}{\mathbb{Z}}
\newcommand{\F}{\mathbb{F}}
\newcommand{\A}{\mathcal{A}}
\newcommand{\powerset}{\mathcal{P}}

\begin{document}
\title{Turing Machines on Graphs and Inescapable Groups}
\author{Aubrey da Cunha}
\begin{abstract}
	We present a generalization of standard Turing machines based on allowing unusual tapes.  We present a set of reasonable constraints on tape geometry and classify all tapes conforming to these constraints.  Surprisingly, this generalization does not lead to yet another equivalent formulation of the notion of computable function.  Rather, it gives an alternative definition of the recursively enumerable Turing degrees that does not rely on oracles.  The definitions give rise to a number of questions about computable paths inside Cayley graphs of finitely generated groups, and several of these questions are answered.
\end{abstract}

\maketitle

\section{Introduction}
	When Alan Turing originally defined his $a$-machines, which would later be called Turing machines, he envisioned a machine whose memory was laid out along a one-dimensional tape, inspired by the ticker tapes of the day.  This seemed somewhat arbitrary and perhaps unduly restrictive, and so, very quickly, machines with multiple and multi-dimensional tapes were proposed.  The focus at the time was defining the term ``computable", and as adding tapes and dimensions defined the same class of functions as Turing's original, simpler model, studying alternate tape geometries fell out of favor for some time.

	The complexity theory community then reignited interest in alternative tape geometries by considering not the class of functions computable by Turing machines, but time and space complexity of functions on different tape geometries.  This led to a number of results about relative efficiency of machines with one/many tapes and one/two/high-dimensional tapes.  For example, the language of palindromes (those strings that read the same forward and backward) can be computed in $O(n)$ time on a two-tape machine, but requires $\Omega(n^2)$ on a one-tape machine with one read/write head.  Or, an $m$-dimensional Turing machine running in time $T(n)$ can be simulated by a $k$-dimensional Turing machine ($k<m$) in time $T(n)^{1+\frac{1}{m}-\frac{1}{k}+\epsilon}$ for all $\epsilon>0$ \cite{Grigorev1979}.

	Many of the proofs and algorithms used in the study of multidimensional Turing machines make their way into or are inspired by the world of mesh-connected systems.  Mesh-connected systems are arrays of identical, relatively dumb processors (typically with memory logarithmic in the input size) that communicate with their neighbors to perform a computation.  Time use on a Turing machine with a $d$-dimensional tape is intimately tied to the power use of a $d$-dimensional mesh of processors with finite memory, so there is some natural crossover.  Mesh-connected systems constitute an area of very active research now, but since it remains very closely tied to the physical implementation, research is generally restricted to two- and three-dimensional grids.

	In this paper, we go beyond the world of rectangular grids and consider tapes at their most general.  The purpose of this is two-fold.  First, to give the complexity theorists a general framework in which to work, subsuming all current tape-based Turing models.  Second, to provide some evidence that alternative tape geometries are interesting from a recursion-theoretic perspective.  Along the way, we will encounter some questions in combinatorial group theory that aren't directly related to generalized Turing tapes, but are interesting in their own right.

	Any Turing tape can be modeled as a digraph with nodes corresponding to tape cells and edges corresponding to allowable transitions.  Hence, we start by introducing a number of graph-theoretic conditions that ought to be satisfied by a reasonable tape.  It turns out that the criteria we outline are necessary and sufficient conditions for the graph to be the Cayley graph of a finitely generated, infinite group.

	We then turn to the question of whether allowing arbitrary Cayley graphs as tapes is just another equivalent machine model for the class of computable functions.  Interestingly, this will depend on the structure of the group from which the Cayley graph is generated.  For groups with solvable word problem, this does indeed lead to machines that compute the class of computable functions, however for groups with unsolvable word problems, these machines are strictly more powerful than standard Turing machines.  In fact, they can be as powerful as any oracle machine and we end up with an alternative definition of the Turing degrees that is machine based and doesn't rely on oracles.

	The constructions and proofs of these results begin to raise questions about what kind of computable objects we can hope to find in arbitrary Cayley graphs.  In particular, whether we can always find an infinite, computable, non-self-intersecting path.  We call such a path an escape and we construct a group without any escapes.  This construction is non-trivial as any group without an escape must also be a Burnside group, which we also prove.

	Throughout the rest of this piece, all Turing machines will have a single tape and a single head.  This is the easiest case to treat and the generalization to multiple tapes and multiple heads is entirely analogous to the standard Turing picture.

\section{General Tape Geometries}
	Any tape essentially consists of a collection of cells, each of which can hold a single symbol, together with a mechanism for moving from one cell to another.  So underlying any tape is an edge-colored digraph.  Edges in this graph represent allowable transitions between states and the edge coloring encodes the conditions on the stored symbol and control state under which that particular transition occurs.  In order to be a reasonable Turing tape, this digraph should satisfy a few restrictions.
	\begin{enumerate}
		\item Uniqueness of outgoing colors - From any vertex, there should be exactly one outgoing edge of each color.  Since the mechanism by which the tape head moves is encoded in the edge color, outgoing edges should have different colors.  Also, since the transition function is independent of the tape cell, the collection of colors going out from each vertex should be the same.
		\item Homogeneity - Every vertex should ``look like" every other vertex.  Technically, this means that the subgroup of the automorphism group of the graph that preserves edge colors should be vertex transitive.  This is an extension of the assumption that all tape cells are indistinguishable, in this case by the local geometry.
		\item Infinity - The tape should have an infinite number of cells.  Otherwise, it's just a finite automaton.
		\item Connectedness - As the head moves during the course of a computation, it remains on a single connected component.  Inaccessible states are useless, so we can require that every tape cell be accessible from the starting point.  In particular, this means that the graph is connected.
		\item Finitely many colors - The transition function of the TM should be finite, so there should only be finitely many outgoing colors.  Having more colors doesn't change the computational power, since only finitely many of them could be referenced by the transition function anyway.
		\item Backtracking - The Turing machine should be able to return to the cell it just came from.  This assumption is less essential, since in view of homogeneity, any algorithm that called for returning to the previous tape cell could be replaced by a fixed sequence of steps.  However, many algorithms call for the head to return to the previous cell and forcing the head to do so by a circuitous route seems unduly harsh.  Note that in view of homogeneity, the color of an edge determines the color of the reverse edge.
	\end{enumerate}

	Restrictions~1 and~2 imply that our tape is a Cayley graph, restrictions~3,~4, and~5 make it the Cayley graph of a finitely generated infinite group, and restriction~6 forces the generating set to be closed under inverses.  In addition, any Cayley graph of of a finitely generated infinite group with a generating set that is closed under inverses satisfies~1--6.  This suggests that Cayley graphs are, in some sense, the ``right" degree of generality and leads to the following definition:

	\begin{definition}
		Let $G$ be an infinite group and $S\subset G$ be a finite generating set for $G$ that is closed under inverses.  Then the Cayley graph of $G$ generated by $S$ is called the {\em tape graph}, $(G,S)$.
	\end{definition}

	Using this general type of tape, we can then ask questions about the structure of Turing Machines with tapes given by assorted groups and generating sets.

\section{Turing Machines on Cayley Graphs}
	\begin{definition}
		Let $G=\langle g_1,\ldots,g_n\rangle$ be a finitely generated group with the set $\{g_1,\ldots,g_n\}$ closed under inverses.  Then a Turing Machine over $G$ with generating set $\langle g_1,\ldots, g_n\rangle$ is a 7-tuple, $(Q,\Gamma,b,\Sigma,\delta,q_0,F)$ where
		\begin{itemize}
			\item $Q$ is the finite set of states
			\item $\Gamma$ is the finite set of tape symbols
			\item $b\in\Gamma$ is a designated blank symbol
			\item $\Sigma\subset\Gamma\backslash\{b\}$ is the set of input symbols
			\item $\delta:Q\times\Gamma\rightarrow Q\times\Gamma\times\{g_1,\ldots,g_n\}$ is the transition function
			\item $q_0\in Q$ is the initial state
			\item $F\subset Q$ is the set of terminal states (typically one to accept and one to reject)
		\end{itemize}
	\end{definition}
	This definition varies from the standard definition only in the interpretation of the transition function.  Whereas a standard TM has a two-way infinite one-dimensional tape and the transition function includes instructions for moving left or right, a TM over $G$ has as a tape the Cayley graph of $G$ and the transition function has instructions for moving along edges labeled by a particular generator.  For example, a Turing Machine over $\Z$ with generating set $\{-1,+1\}$ is a standard one-dimensional TM and a TM over $\Z^2$ with generating set $\{(0,-1),(-1,0),(0,1),(1,0)\}$ is a standard two-dimensional TM.

	We have intentionally skipped the notion of how to provide input for machines of this type.  Most generally, we could insist that the initial set of non-blank tape cells be connected and contain the initial location of the tape head.  However, we really intend these machines to work like Turing machines, and therefore, to compute on strings of symbols.  It turns out that there will be a canonical way to lay out strings on the Cayley graph, but we need some machinery first.

\subsection{A Well-ordering in Trees}\label{treeOrdering}
	We shall turn aside from the main topic for a moment to discuss a general statement about trees.  There are many ways to define an order on the vertices of a tree, but we are going to be interested in the lexicographic order.  In general, lexicographic orderings on trees have few nice properties, but we show that finitely branching trees have a subtree where the lexicographic order is in fact, a well-order.

	First, some definitions.  Let $T$ be a finitely branching tree.  Denote by $[T]$ the set of all infinite paths through $T$ and by $\sqsubset$ the partial ordering on vertices induced by the tree.  Our convention will be that $v\sqsubset w$ means that $v$ is closer to the root than $w$.

	In order for the lexicographic ordering on $T$ to even make sense, we must have a linear order on the set of children of each node.  Denote the order on the children of $v\in T$ by $<_v$.  Then the lexicographic order, $<$ is defined as follows:
	\begin{itemize}
		\item If $v\sqsubset w$, then $v<w$.
		\item If neither $v\sqsubset w$ nor $w\sqsubset v$, let $u$ be the greatest lower bound of $v$ and $w$ according to $\sqsubset$ and let $u'\sqsubset v$ and $u''\sqsubset w$ be children of $u$.  Then $v<w\iff u'<_u u''$.
	\end{itemize}
	This order can, in fact, be extended to an order on $T\cup [T]$.  Identifying elements of $[T]$ with subsets of $T$ and elements of $T$ with one-element subsets of $T$, we can define
	\begin{equation*}
		x<y\iff (\exists w\in y)(\forall v\in x)v<w
	\end{equation*}

	Defined in this way, $<$ is a linear order, but we can't really hope for any more structure than that.  But, as promised, with a bit of pruning, we can find a subtree with much more structure.

	\begin{theorem}\label{thmWellOrder}
		Let $T$ be a finitely branching tree and let $<$ be the lexicographic ordering on the nodes of and paths through $T$ as given above. Define 
		\begin{equation*}
			T'=\{v\in T|\forall w\in[T],v<w\}
		\end{equation*}
		Then $<$ restricted to $T'$ is order isomorphic to an initial segment of $\omega$.  In addition, if $T$ is infinite, $T'$ is infinite as well.
	\end{theorem}
	\begin{proof}
		This follows from the following direct result of K\"onig's Lemma.
		\begin{lemma}
			 Every element of $T'$ has only finitely many $<$-predecessors.
		\end{lemma}
		\begin{proof}
			Suppose $v\in T'$ had infinitely many $<$-predecessors.  Then we could form the tree,
			\begin{equation*}
				S=\{w\in T'|w<v\}
			\end{equation*}
			This is, in fact, a tree since $T'$ is a tree and $x\sqsubset y$ implies $x<y$.  Since $v$ has infinitely many $<$-predecessors, $S$ is infinite.

			By K\"onig's Lemma, there must be a path through $S$, call it $P$.  But $T'$ is a subtree of $T$ so $P\in [T]$.  By definition of $<$, $P<v$, but $v\in T'$ so $v<P$.  This is a contradiction, so $v$ must have only finitely many predecessors.
		\end{proof}
		Any linear order in which every element has only finitely many predecessors clearly cannot have an infinite descending chain, so must be a well-order.  As $\omega+1$ has an element with infinitely many predecessors, the order type must be an initial segment of $\omega$.

		For the second part of Theorem~\ref{thmWellOrder}, we need an additional lemma.
		\begin{lemma}\label{lmMinPath}
			If $[T]$ is non-empty, then $[T]$ has a minimal element in the $<$ ordering.
		\end{lemma}
		\begin{proof}
			We can inductively construct the minimal element of $[T]$.  For any $v\in T$, define $s(v)$ to be the minimal (according to $<_v$) child of $v$ that is a member of some element of $[T]$ if such a vertex exists.  Note that if there is a path through $v$, $s(v)$ is defined and there is a path through $s(v)$.  If $v_0$ is the root, then $P=\{s^{(n)}(v_0)\}_{n\in\N}$ is the desired minimal element.

			Since $[T]$ is non-empty, there is a path through the root and so, by induction $s^{(n)}(v_0)$ is defined for all $n$.  Therefore $P$ is indeed a path.  

			To see that $P$ is minimal, let $P\neq P'\in[T]$.  Let $v=s^{(m)}(v_0)$ be the largest element (according to $\sqsubset$) of $P\cap P'$ and let $w\in P'$ be a child of $v$.  By maximality of $v$, $w\neq s(v)$ and by construction, $s(v)\leq_v w$.  Therefore, $s(v)<_v w$.  By the lexicographic ordering, $w$ is greater than all descendants of $s(v)$ and also greater than all ancestors of $s(v)$ (since ancestors of $s(v)$ are also ancestors of $w$).  Therefore, $w>u$ for all $u\in P$ and $P'>P$.
		\end{proof}

		Now, let $T$ be infinite.  Then, by K\"onig's Lemma again, $[T]$ is non-empty.  Let $P$ be the minimal path in $[T]$ according to Lemma~\ref{lmMinPath}.  Then $P\subset T'$ since for any $v\in P$ and $P'\in [T]$, $v<P\leq P'$.  $P$ is infinite, so $T'$ is infinite as well.
	\end{proof}

\subsection{Power of Turing Machines on Cayley Graphs}
	One of the first questions to be asked about any new model of computation is whether the class of functions computable by the new model is different from the class of computable functions.  For Turing machines on Cayley graphs, this depends rather sensitively on properties of the group producing the tape graph.  For example,
	\begin{lemma}\label{lemWordProblem}
		Let $(G,S)$ be a tape graph.  There is a Turing machine over $(G,S)$ that can solve the word problem for $G$.
	\end{lemma}
	This will not be proved rigorously, since we have not yet defined how input is to be provided, but we will provide an argument that can be made rigorous in an obvious fashion by the end of this section.

	Given two sequences of generators, the machine can simply follow the first sequence of generators, leaving a pointer at each cell along the way pointing to its predecessor.  Marking the end, it can follow the sequence of pointers back to the origin.  Now, it can follow the second sequence of generators and check to see whether the end point was marked in the first step.  Clearly, if this algorithm ends on the marked cell, the two sequences of generators correspond to the same group element.

	Boone and Novikov \cite{Boone1958,Novikov1958} independently showed that there exist groups with undecidable word problems, so this leads us to believe that Turing machines on a given group with unsolvable word problem are strictly more powerful than standard Turing machines.  However, this requires that Turing machines over said group also be able to compute all computable functions.  Fortunately, this is the case.

	\begin{theorem}\label{thmCayleySimulation}
		Let
		\begin{equation*}
			M=(Q_M,\Gamma_M,b_M,\Sigma_M,\delta_M,{q_0}_M,F_M)
		\end{equation*}
		be a standard one-dimensional one-tape Turing Machine and let $(G,S)$ be a tape graph.  Then there is a Turing Machine over $(G,S)$ that can simulate $M$.
	\end{theorem}

	The simulation itself is very straight-forward.  The only difficulty stems from the question of how to arrange the tape contents of the simulated machine on the Cayley graph.  If we could compute an infinite non-self-intersecting path through the Cayley graph, we could use this as a standard one-dimensional tape and do the simulation there.  However, as we will see in Section~\ref{secInescapable}, such a path need not exist.

	Fortunately, we can do the simulation anyway, in this case, by a variant of the ``always turn left" algorithm for solving mazes.  By putting an ordering on the generators of our tape graph, ``always turn left" becomes ``always follow the lexicographically minimal edge".  So, in an appropriate tree, Section~\ref{treeOrdering} gives us a well-ordered subtree where we can do the simulation with the $n$th vertex in the well-ordering storing the contents of the $n$th tape cell.

\subsection{Proof of Theorem \ref{thmCayleySimulation}}
	We will start with a description of the tree where we will store the tape of the simulated machine.  After we describe the operation of the machine, we will prove that this tree can be constructed on-line and navigated effectively.

	We will routinely use the natural correspondence between sequences of generators and group elements given by forming the product of the generators in the given sequence and evaluating in the group.  Henceforth, sequences and group elements will be used interchangeably.  Of course, multiple sequences will correspond to the same group element, but the sequence should always be clear from context.

	\begin{definition}
		Let $G$ be a group.  A {\em super-reduced word} is a finite sequence of elements of $G$ such that no subword, taken as a product in $G$ is equal to the identity.  More precisely, it is a sequence, $g_1,\ldots,g_n$ such that for all $1\leq i<j\leq n$, $\prod_{k=i}^jg_k\neq e$ in $G$.
	\end{definition}
	In the context of tape graphs, super-reduced words with symbols from the generating set correspond exactly to non-intersecting finite paths through the tape graph.  Note that a word is super-reduced if and only if all of its prefixes are super-reduced.

	Form the tree, $T$, of super-reduced words with symbols from $S$.  Since every group element has at least one super-reduced word corresponding to it and $G$ is infinite, $T$ is a spanning tree for the Cayley graph of $G$, hence, infinite.  Therefore, we can construct an infinite $T'$ as in Section~\ref{treeOrdering} where the lexicographic ordering is a well-ordering.  We will not do the computation on $T'$, but on a subtree, which we will call $R$.  To define $R$ we will want another definition.

	\begin{definition}
		We say that two sequences of generators, $v$ and $w$, are equivalent in $G$, or $v\equiv_G w$ if they correspond to the same group element.  In other words, $v\equiv_G w$ if $v=(s_1,\ldots,s_n)$, $w=(r_1,\ldots,r_m)$ and
		\begin{equation*}
			\prod_{i=1}^ns_i=\prod_{i=1}^mr_i\,\,\,\text{(in }G\text{)}
		\end{equation*}
	\end{definition}

	Now we can define $R$ as follows,
	\begin{equation*}
		R=\{v\in T'|(\forall w\in T')\,v\equiv_G w\implies v<w\}
	\end{equation*}
	That is, $R$ is the set of vertices in $T'$ that are lexicographically minimal among sequences that represent the same group element.

	It's not obvious at first glance, but $R$ is a tree.  Suppose to the contrary that $uv= w\in R$ but $u\notin R$.  Then there is some $u'<u$ in $R$ corresponding to the same group element.  Since we began with the tree of super-reduced words, $u$ and $u'$ are incomparable in the tree order.  Therefore, $u$ and $u'$ must differ at some first location.  So, $u'v\equiv_G w$ but $uv$ and $u'v$ differ for the first time at the same location and since $u'<u$, $u'v<uv=w$.  This is a contradiction since $w$ was supposed to be minimal among words corresponding to the same group element, so $R$ is indeed a tree.

	Also non-obvious is the fact that $R$ is infinite.  By the proof of Theorem~\ref{thmWellOrder}, $T'$ contains a path.  As was noted earlier, since we began with the tree of super-reduced words, elements that are comparable in the tree order cannot correspond to the same group element.  Therefore, the path in $T'$ corresponds to an infinite collection of group elements.  $R$ represents the same group elements since we only pruned redundant representations, so $R$ also represents an infinite number of group elements and is therefore infinite.

	Now for any tape graph, we have a tree that is well-ordered lexicographically, represents infinitely many group elements and represents each individual element at most once.  This is where we will do the simulation.

	Let $(G,S)$ be the tape graph given in the statement of the theorem.  Denote the elements of $S$ by $g_1,\ldots,g_n$.  It will be convenient to define a set $S'=S\cup\{g_0,g_{n+1}\}$ with the ordering $g_0<g_1<\cdots<g_n<g_{n+1}$.  We will consider both $g_0$ and $g_{n+1}$ equal to $e$ in the group for the purposes of moving from node to node.  Define a Turing Machine, $N$, over $(G,S)$ as follows:
	\begin{itemize}
		\item $Q_N=Q_M\sqcup(Q_M\times S')\sqcup(Q_M\times S')\sqcup(Q_M\times S)\sqcup(Q_M\times S)$
		\item $\Gamma_N=\Gamma_M\times S'\times\powerset(S)\times\powerset(S)$
		\item $b_N=(b_M,g_0,\emptyset,\emptyset)$
		\item $\Sigma_N=\Sigma_M\times S'\times\powerset(S)\times\powerset(S)$
		\item ${q_0}_N={q_0}_M$
		\item $F_N=F_M\times S'\times\powerset(S)\times\powerset(S)$
	\end{itemize}

	Each state in the tape alphabet will have an intended meaning.  Remember that we are going to do the computation on a tree, so we have to encode in each node not just the symbol of the simulated machine, but also auxiliary information about the structure of the tree.  In particular, if $(\gamma,\sigma,A,B)\in\Gamma_N$, $\gamma$ is the symbol of the simulated machine stored at the node, $\sigma$ is the generator to follow to reach the ancestor of this node in the tree, $A$ is the set of generators corresponding to edges pointing away from the root in the tree, and $B$ is the set of generators defining non-edges of the tree.  Remember that we are going to be constructing this tree on the fly and so we don't have complete information about which generators correspond to edges of the tree at every step.  Thus, elements of $S\backslash(A\cup B)$ are the edges whose membership in the tree has not yet been determined.

	We will also give each state a name and an intended interpretation,
	\begin{itemize}
		\item $Cq$ for $q\in Q_M$: We are simulating the computation of $M$ and the current state of $M$ is $q$.
		\item $Rqx$ for $q\in Q_M$ and $x\in S'$: The simulated tape head is moving to the right and $M$ is currently in state $q$.  The argument $x$ encodes the edge we followed to reach our current location.
		\item $Lqx$ for $q\in Q_M$ and $x\in S'$: The simulated tape head is moving to the left and $M$ is currently in state $q$.  The argument $x$ encodes the edge we followed to reach our current location.
		\item $Eqx$ for $q\in Q_M$ and $x\in S$: The simulated tape head is moving to the right and $M$ is in state $q$, but we have run out of tape and are attempting to extend the tree along edge $x$.
		\item $Bqx$ for $q\in Q_M$ and $x\in S$: $M$ is currently in state $q$ and we just failed to extend the tree along edge $x$, so we are backtracking.
	\end{itemize}
	Note that the starting state of $N$ is named $C{q_0}_M$.

	It will also be convenient to talk about the component functions of the transition function of $M$, \begin{eqnarray*}
		\delta_1:Q_M\times\Gamma_M&\rightarrow& Q_M\\
		\delta_2:Q_M\times\Gamma_M&\rightarrow&\Gamma_M\\
		\delta_3:Q_M\times\Gamma_M&\rightarrow&\{L,R\}
	\end{eqnarray*}

	We can now write down the action of the transition function.  If the current symbol being read is $(\gamma,\sigma,A,B)$, then the value of the transition function on each type of state is given in Table~\ref{tblSimTransition}.  Entries in the table are triples in the order, new state, symbol written, direction the tape head moves.
	\begin{table}[htbp]
		\begin{tabular}[htb]{l|c|c}
			$\digamma\in\Gamma_N$ & $\delta_N(\digamma,(\gamma,\sigma,A,B))$ \\ \hline
			$Cq$ & $\begin{cases}\left(R\delta_1(q,\gamma)g_0,(\delta_2(q,\gamma),\sigma,A,B),g_0\right)\text{ if }\delta_3(q,\gamma)=R\\ \left(L\delta_1(q,\gamma)\sigma,(\delta_2(q,\gamma),\sigma,A,B),\sigma\right)\text{ if }\delta_3(q,\gamma)=L\end{cases}$\\
			$Lqx$ & $\begin{cases}\left(Cq,(\gamma,\sigma,A,B),g_0\right)\text{ if }\forall y\in A,y\geq x\\ \left(Lqg_{n+1},(\gamma,\sigma,A,B),\max_{y\in A,y<x}y\right)\text{ otherwise}\end{cases}$\\
			$Rqx$ & $\begin{cases}\left(Cq,(\gamma,\sigma,A,B),\min_{y\in A,y>x}y\right)\text{ if }\exists y\in A,y>x\\ \left(Eqy,(\gamma,\sigma,A\cup\{y\},B),y\right)\text{ where }y=\min_{z\in S\backslash B,z>x}z\\ \qquad\qquad\qquad\qquad\qquad\qquad\qquad\qquad \text{ if }\exists z\in S\backslash B,z>x\\ \left(Rq\sigma,(\gamma,\sigma,A,B),\sigma\right)\text{ otherwise}\end{cases}$\\
			$Eqx$ & $\begin{cases}\left(Cq,(\gamma,x^{-1},\emptyset,\{x^{-1}\}),g_0\right)\text{ if }A=B=\emptyset\\ \left(Bqy,(\gamma,\sigma,A,B),x^{-1}\right)\text{ otherwise}\end{cases}$\\
			$Bqx$ & $\left(Rqx,(\gamma,\sigma,A\backslash\{x\},B\cup\{x\}),g_0\right)$
		\end{tabular}
		\caption{Transition Table for $N$}
		\label{tblSimTransition}
	\end{table}

	Most of these transitions are pretty opaque, so some explanation is warranted.

	If we are in a $C$-state, then we perform one step of the computation and then transition into either an $R$-state or and $L$-state depending on whether the computation tries to move left or right.  A leftward step in the simulated machine always begins with a step toward the root for the simulating machine, so we take that step immediately.  Rightward steps are more complicated, so we leave the tape head where it is on a rightward step.

	Taking a step to the left, we want to find the lexicographic immediate predecessor of our current vertex.  We begin by taking one step toward the root, which we did when we transitioned out of the $C$-state.  Then, either there is a branch all of whose elements are less than where we started, or not.  If not, then we are at the immediate predecessor of our origin, and we continue computing.  Otherwise, follow the greatest such branch and always move to the greatest child until we reach a dead-end.  This dead-end is the immediate predecessor we were looking for, so we can continue the computation.

	Taking a step to the right is significantly more complicated, as we are most likely going to have to extend the tree as we go.  If we have already built some tree above us ($A$ is non-empty), then we simply use that part of the tree, moving along the minimal edge in the tree and continuing the computation.  This is the reason for using the ``false" generator $g_0$ when we transition out of the $C$-state.  Otherwise, try to extend the tree along the least edge that we have not already ruled out.  We add that edge to $A$ and switch to an $E$-state.

	If $A$ and $B$ are both empty, then this is a new vertex that we have not visited before, so continue with the computation.  Otherwise, we are at a vertex that has already been added elsewhere in the tree.  So, we back up and transition to a $B$-state.  In the $B$-state, we rule out the edge we just took by removing it from $A$ and adding it to $B$ and switch to an $R$-state to try extending the tree again.

	If we can't extend the tree at all ($B=\{g_1,\ldots,g_n\}$), then take a step toward the root and try again, eschewing anything less than or equal to the edge we backtracked along.  Since there is an infinite tree for us to use, we will eventually be able to extend the tree.

	We are now in a position to be more explicit about input.  Take a machine with one tape given by a tape graph, $(G,S)$ and another, read-only one-way standard input tape.  Then, in view of the construction above, it is clear how this machine would transcribe its input from the standard tape onto the tape graph.  Once the transcription is done, the computation can proceed according to the above construction.  Thus, it seems reasonable to require the input to a machine be the result of a transcription of this type.  This does require some (potentially) non-recursive manipulation of a string to produce the appropriate input, but if this offends you, it is always possible to go back to the formalism with an auxiliary read-only input tape.

\subsection{An Alternative Characterization of the Turing Degrees}\label{secTDegrees}
	We have demonstrated that Turing Machines on arbitrary Cayley graphs are strictly more powerful than standard Turing Machines, so the next question to ask is, ``how much more powerful?"  The short answer is ``as powerful as we want".  In \cite{Boone1966a,Boone1966b} Boone showed that for any r.e. Turing degree, there is a finitely presented group whose word problem is complete for that degree.  Using such a group, we can produce a machine (more precisely a class of machines) that computes exactly the functions in or below the given degree.  More precisely,
	\begin{theorem}
		Let $T$ be an r.e. Turing degree.  There is a group, $G$, such that the class of functions computable by a Turing Machine over $G$ is exactly the class of functions in or below $T$.
	\end{theorem}
	\begin{proof}
		By Boone, let $G$ be a group whose word problem is complete for $T$ and let $f$ be a function in or below $T$.  Since $f\leq T$, $f$ is Turing reducible to the word problem for $G$.  Turing machines over $G$ can perform this reduction since they can compute all recursive functions and they can solve the word problem for $G$ by Lemma~\ref{lemWordProblem}.  Therefore, they can compute $f$.

		To show that a Turing Machine over $G$ cannot compute a class larger than $T$, observe that a Turing Machine with an oracle for the word problem for $G$ can easily simulate a Turing Machine over $G$.  It simply maintains a list of nodes written as a sequence of generators for the address together with whichever tape symbol is written there.  When the simulated tape head moves, the machine simply appends the generator to the current address and consults the oracle to determine which node this corresponds to, adding a new entry if the new node isn't in the list.
	\end{proof}

\section{Inescapable Groups}\label{secInescapable}
\subsection{Construction of an Inescapable Group}
	Theorem~\ref{thmCayleySimulation} raises the question of whether or not a Turing Machine can always walk off to infinity on the Cayley graph of an infinite group without retracing its steps.  This suggests the following definition:
	\begin{definition}
		An {\em inescapable group} is a tape graph, $(G,S)$, such that any infinite computable sequence, $s$, of elements from $S$ corresponds to a self-intersecting path.
	\end{definition}

	This definition leads to a number of questions.  Do such things exist?  What do they look like?  Is inescapability a group invariant?  Very little is known in response.  For example, as we will see in Section~\ref{secBurnside}, an inescapable group must be a Burnside group; that is, every element must have finite order.  The purpose of this section is to answer the more fundamental question of whether or not inescapable groups exist.

\subsubsection{Definitions}
	We will need to make a few definitions.  Let $A$ be a finite set.  Let $A^*=A^{<\omega}$ be the set of all finite sequences with elements from $A$, called words over $A$.  Let $\epsilon$ denote the empty word and let $A^+=A^*\backslash\{\epsilon\}$ be the set of non-empty words.  If $w\in A^*$, denote by $|w|$ the length of~$w$, by $w(i)$ the $i$th symbol in~$w$ (indexed from~0), and by $w(i,j)$ the word $w(i)w(i+1)\ldots w(j)$.  Also, let $A^{\leq n}=\{w\in A^*,|w|\leq n\}$ be the set of words of length no greater than~$n$.

	For $w,w'\in A^*$ define the following relations:
	\begin{itemize}
		\item $w'$ is a subword of $w$ if there exist $0\leq i,j<|w|$ such that $w'=w(i,j)$.
		\item $w'$ is a subsequence of $w$ if there exist $0\leq i_1<i_2<\ldots <i_{|w'|}<|w|$ such that $w'=w(i_1)w(i_2)\ldots w(i_{|w'|})$.
		\item $\#(w,w')=\left|\{x\in\N^{|w'|}|x_0<\ldots<x_{|w'|-1}\text{ and }w(x_i)=w'(i)\text{ for all }i\}\right|$ is the number of ways in which $w'$ is a subsequence of $w$.
	\end{itemize}
	Notice that $\epsilon$ is both a subword and subsequence of every word and that $\#(w,\epsilon)=1$ for all $w$.  As an example, note that while $aabbaa$ does not contain $aba$ as a subword, it does contain $aba$ as a subsequence in 8 different ways.

	We will say that an infinite sequence of symbols from $A$, call it $s$, is computable if there is a Turing Machine which, on input $n$, produces $s(n)$.

	When dealing with sets, we will use $\powerset(X)$ to denote the collection of all subsets of $X$ and $\powerset_r(X)$ to denote all subsets of $X$ of cardinality $r$.

	We will also be dealing with polynomials in a non-commutative polynomial ring, so let $f\in K\langle x_1,\ldots,x_d\rangle$ be a polynomial over a field, $K$, with non-commuting indeterminates.  We say that $f$ is homogeneous of degree~$n$ if $f$ is a $K$-linear combination of monomials of the form
	\begin{equation*}
		x_{i_1}^{n_1}x_{i_2}^{n_2}\cdots x_{i_k}^{n_k}
	\end{equation*}
	with $n_1+n_2+\ldots+n_k=n$.  In this case, we write $\partial(f)=n$.

\subsubsection{A Combinatorial Lemma}
	When we get to the actual construction, we are going to need to find a subsequence of each computable sequence that satisfies certain properties.  In particular, we need the following lemma.

	\begin{lemma}\label{lmEvenSubstrings}
		Let $A=\{x_1,\ldots,x_m\}$ and let $n\geq 1$.  There is some $C(m,n)$ such that for all words, $s$, over alphabet $A$ with $|s|\geq C(m,n)$, $s$ contains a non-empty subword, $w$, with the following property.  For each $w'\in A^{\leq n}\cap A^+$, $w$ contains $w'$ as a subsequence an even number of times.
	\end{lemma}

	In fact, the number $C(m,n)$ is a Ramsey number, $R(2,3,2^\frac{m^{n+1}-1}{m-1})$ where the function $R(r,k,n)$ is given in Ramsey's Theorem,
	\begin{theorem}[Ramsey]
		Let $r,k,n$ be positive integers with $1\leq r\leq k$.  Then there exists an integer, denoted $R(r,k,n)$, such that for each set $X$ with $|X|=R(r,k,n)$ and each partition of $\powerset_r(X)$, $Y_1,\ldots,Y_n$, there exists a $k$-element subset $Y$ of $X$ and a set $Y_i$ with $\powerset_r(Y)\subset Y_i$.
	\end{theorem}

	To prove the lemma, we need a theorem from the combinatorial theory of words.
	\begin{theorem}[Pirillo \cite{Pirillo1983}]\label{thmPirillo}
		Let $\phi:A^+\rightarrow E$ be a mapping from $A^+$ to a set $E$ with $|E|=n$.  For each $k\geq 1$, each word $w\in A^+$ of length $R(2,k+1,n)$ contains a subword $w_1w_2\ldots w_k$ with $w_i\in A^+$ and 
		\begin{equation*}
			\phi(w(i,i'))=\phi(w(j,j'))
		\end{equation*}
		for all pairs $(i,i')$, $(j,j')$ ($1\leq i\leq i'\leq k$ and $1\leq j\leq j'\leq k$).
	\end{theorem}

	\begin{proof}[Proof of Lemma \ref{lmEvenSubstrings}]
		Consider the function $\phi:A^+\rightarrow \Z_2^{A^{\leq n}}$ defined as follows:
		\begin{equation*}
			\left(\phi(w)\right)(w')=\#(w,w')\bmod{2}
		\end{equation*}
		Since
		\begin{equation*}
			\left|\Z_2^{A^{\leq n}}\right|=2^{\frac{m^{n+1}-1}{m-1}}
		\end{equation*}
		We can apply Theorem~\ref{thmPirillo} with $k=2$ to $s$ to get $w_1w_2$, a subword of $s$ such that
		\begin{equation*}
			\phi(w_1)=\phi(w_2)=\phi(w_1w_2)
		\end{equation*}
		Then the word $w_1w_2$ contains any non-empty sequence over $A$ of length no greater than $n$ as a subsequence an even number of times.

		We will prove something slightly stronger by induction on the length of the contained subsequence.  Specifically, we will show that not only does $w_1w_2$ satisfy the theorem, but so do $w_1$ and $w_2$ individually.

		For the base case, let $x_i$ be a subsequence of length one.  Then
		\begin{equation*}
			\phi(w_1w_2)(x_i)=\phi(w_1)(x_i)+\phi(w_2)(x_i)=2\phi(w_1)(x_i)=0\bmod{2}
		\end{equation*}
		since the number of occurrences of a single symbol in a concatenation of words is simply the sum of the number of occurrences in each factor.  Using the fact that $\phi(w_1w_2)=\phi(w_1)=\phi(w_2)$, we have established the base.

		The induction step isn't significantly more difficult; the only difficulty arises from the fact that the number of occurrences of a substring of length greater than one isn't additive.  However, we do have the formula,
		\begin{equation*}
			\#(w_1w_2,w')=\sum_{i=0}^{|w'|}\#(w_1,w'(0,i-1))\#(w_2,w'(i,|w'|-1))
		\end{equation*}
		We really only care about the parity of this expression and if we assume the induction hypothesis for all strings shorter than $w'$, most of the terms are even.  So by reducing,
		\begin{equation*}
			\#(w_1w_2,w')=\#(w_1,w')+\#(w_2,w')\bmod{2}
		\end{equation*}
		Now, in an argument analogous to the base case, we get 
		\begin{equation*}
			\phi(w_1w_2)=\phi(w_1)=\phi(w_2)=0\bmod{2}
		\end{equation*}
	\end{proof}

\subsubsection{The Golod-Shafarevich Theorem}
	The Golod-Shafarevich Theorem is a powerful tool from algebra that gives a sufficient condition for a particular quotient algebra to be infinite dimensional.  For our purposes, it is a tool that will ensure that as we start adding relations to a free group, we don't collapse the group to something finite.

	The power of the theorem comes from the fact that the criterion it presents is based only on the number of relations of certain types, and not on the relations themselves.  This gives us a large amount of freedom to choose the relations we want without having to worry about bad interactions between them.

	\begin{theorem}[Golod-Shafarevich \cite{GolodShafarevich}]\label{thmGS}
		Let $R_d=K\langle x_1,\ldots, x_d\rangle$ be the polynomial ring over a field, $K$, in the non-commuting indeterminates $x_1,\ldots,x_d$.  Let $f_1,f_2,\ldots\in F$ be a set of homogeneous polynomials of $R_d$, and let the number of polynomials of degree $i$ be $r_i$.  Let $2\leq\partial(f_i)\leq\partial(f_{i+1})$ and let $I$ be the ideal generated by $F$.  Let $R_d/I=A$.  If all the coefficients in the power series,
		\begin{equation*}
			\left(1-dt+\sum_{i=2}^\infty r_it^i\right)^{-1}
		\end{equation*}
		are non-negative, then $A$ is infinite dimensional.
	\end{theorem}

	In a subsequent paper, Golod \cite{Golod1964} proves the following corollary,
	\begin{corollary}\label{corDimBound}
		In Theorem~\ref{thmGS}, if
		\begin{equation*}
			r_i\leq\epsilon^2(d-2\epsilon)^{i-2}
		\end{equation*}
		where $0<\epsilon<\frac{d}{2}$, then $A$ is infinite-dimensional.
	\end{corollary}
	For example, taking $d=2$ and $\epsilon=\frac{1}{4}$ in the corollary, we see that if $r_i\leq 2$ for all $i\geq 11$ and $r_i=0$ for all $i<11$, $A$ is infinite dimensional.

	Golod used this fact to establish the existence of a Burnside group, that is, an infinite group in which every element has finite order.  He did more than this, in fact, and produced an infinite $p$-group for each prime, $p$.  The diagonalization in Section~\ref{secDiag} will follow the same general ideas as Golod's construction as simplified for countable fields in Fischer and Struik \cite{FischerStruik}.

\subsubsection{The Construction}\label{secDiag}
	In Fischer and Struik \cite{FischerStruik}, there is a construction of a nil-algebra over finite and countable fields.  Although not expressly stated there, the construction is essentially a diagonalization over all polynomials.  In fact, even Golod's original construction of a nil-algebra can be viewed as a diagonalization over all polynomials, but the fact that the collection of all polynomials in the general case is uncountable makes it more difficult to see.

	The presentation given here will follow the construction in Fischer and Struik, since we will be diagonalizing against the set of computable sequences, which is countable.  Thus, we will use the more straightforward construction.

	\begin{theorem}\label{thmInescGroup}
		There exists an inescapable group.
	\end{theorem}

	\begin{proof}
		Let $d\geq 2$ and $A=\{x_1,\ldots,x_d\}$.  Consider the algebra, $\A=\F_2\langle A\rangle$ in non-commuting indeterminates, $x_i\in A$.  Let 
		\begin{equation*}
			S=\bigcup_{i=1}^d\{(1+x_i),(1+x_i)^{15}\}
		\end{equation*}
		and enumerate all c.e. sequences over $S$: $s_0,s_1,\ldots$.  We will use the elements of $S$ interchangeably as characters in an alphabet and as polynomials in $\A$.

		We will construct a set of homogeneous polynomials, $F$, such that the following conditions hold:
		\begin{itemize}
			\item $x_i^{16}\in F$ for $1\leq i\leq d$.  (This will ensure that $(1+x_i)$ has order~16 and is therefore multiplicatively invertible in the quotient algebra)
			\item For each $i\in\N$, $s_i$ has a subword, $w$ such that
				\begin{equation*}
					\left(\prod_{j=0}^{|w|-1}w(j)\right)-1
				\end{equation*}
				is in the ideal generated by elements of $F$.
			\item The number of elements of $F$ of degree $i$ is~0 for $i<16$, $d$ for $i=16$, and either~0 or~1 for every $i>16$.
		\end{itemize}

		Define the following sequence recursively:
		\begin{eqnarray*}
			r_0&=&16\\
			r_{n+1}&=&15\cdot R\left(2,3,2^{\frac{|S|^{r_{n}+1}-1}{|S|-1}}\right)
		\end{eqnarray*}
		We can now enumerate the elements of $F$ as follows.

		Start with $F=\{x_1^{16},\ldots,x_d^{16}\}$ and begin enumerating the elements of all $s_i$ in parallel.  If, at any point, we have enumerated a contiguous subsequence of elements of some $s_i$ of length $r_{i+1}$, call it $v$ and do the following.

		By Lemma~\ref{lmEvenSubstrings}, we can find a non-empty subword, $w$, of $v$ such that $|w|\leq \frac{1}{15}r_{i+1}$ and every sequence of length $\leq r_i$ occurs an even number of times in $w$.  Then, by multiplying out,
		\begin{equation*}
			p(x_1,\ldots,x_d):=\prod_{k=0}^{|w|-1}w(k)=\sum_{w'\in S^*}\#(w,w')\prod_{k=0}^{|w'|-1}(w'(k)-1)
		\end{equation*}
		Note that if $0<|w'|\leq r_i$, $\#(w,w')\equiv 0\bmod{2}$ and that the constant term of $p$ is~1.  Note also that $w'(k)-1$ always has constant term~0, so $p-1$ has no non-zero terms with degree $\leq r_i$.  Since $p$ is a polynomial, we can write $p-1$ as a sum of homogeneous components, $f_1,\ldots,f_m$.  Note that for all $1\leq i\leq m$, $r_i<\partial(f_i)\leq 15|w|\leq r_{i+1}$.  Add all of these to $F$, stop enumerating $s_i$ but continue enumerating everything else that hasn't been similarly halted.

		The $F$ that is so constructed clearly contains $x_i^{16}$ for $1\leq i\leq d$ and for each $s_i$, if $s_i$ is total, it has a subword whose corresponding product is equivalent to~1 mod the ideal generated by $F$.  Also, each $s_i$ only adds polynomials to $F$ that have degree in the interval $(r_i,r_{i+1}]$ and adds at most one polynomial of any given degree.  Since we start with $d$ polynomials of degree~16 and $r_0=16$, the number of elements of $F$ of degree $i$ must be~0 for all $i<16$ and either~0 or~1 for every $i>16$.

		We would then like to show that the multiplicative semigroup of $\A/(F)$ generated by elements of $S$, call it $G$, is an inescapable group.  First, note that by the binomial theorem,
		\begin{equation*}
			(1+x_i)(1+x_i)^{15}=1+x_i^{16}=1
		\end{equation*}
		so $G$ is a genuine group and $S$ is closed under inverses.  $S$ trivially generates $G$, so we only need $G$ to be infinite for $(G,S)$ to be a tape graph.

		By elementary calculus,
		\begin{equation*}
			2\leq d\leq\frac{1}{16}(d-.5)^{i-2}
		\end{equation*}
		for all $i\geq 11$.  Therefore, $F$ satisfies the hypotheses of Corollary~\ref{corDimBound} and $\A/(F)$ is infinite-dimensional.

		Now, for any positive integer, $d$, there must be a monomial of degree $d$ that does not lie in the ideal generated by $F$.  Otherwise, every monomial of degree~$\geq d$ would be in the ideal and the quotient algebra would be finite-dimensional.  So, consider two generic such monomials,
		\begin{equation*}
			x_{i_1}x_{i_2}\ldots x_{i_M}\text{ and }x_{j_1}x_{j_2}\ldots x_{j_N}
		\end{equation*}
		with $M>N$ and the group elements,
		\begin{equation*}
			u=(1+x_{i_1})(1+x_{i_2})\ldots(1+x_{i_M})\text{ and }v=(1+x_{j_1})(1+x_{j_2})\ldots(1+x_{j_N})
		\end{equation*}
		Then,
		\begin{equation*}
			u-v=x_{i_1}x_{i_2}\ldots x_{i_M}+\ldots
		\end{equation*}
		where the remaining terms all have degree $\leq M$.  Since the ideal $(F)$ is generated only by homogeneous polynomials, $u-v\in (F)$ if and only if every homogeneous component of $u-v$ is in $(F)$.  However, the degree $M$ component of $u-v$ is clearly not in $(F)$, so neither is $u-v$.  Therefore, $u$ and $v$ are different elements in $G$.

		Thus, we can find infinitely many distinct elements of $G$, one for each degree.  Therefore, $G$ is infinite and $(G,S)$ is a tape graph.

		In addition, any computable sequence of generators is one of the $s_i$, so we have ensured that it has a subword such that the corresponding product is equal to~1 in the quotient algebra.  This is the same as having product~1 in the group, so all computable sequences of generators must correspond to self-intersecting paths.
	\end{proof}

	The construction given above in fact does better than producing an inescapable group.  Since every step of the construction can be done recursively, the set of relations in the group is r.e.  It is a standard result that a group with an r.e. set of relations is recursively presentable, so the construction produces a {\em recursively presentable} inescapable group.  That being said, the question of whether there exists a finitely presentable inescapable group remains open.  It is also unlikely that the word problem for the group constructed above is solvable, so there also remains the question of whether there exists an inescapable group with solvable word problem.  It should also be noted that the only property of computable sequences that the construction used was that there are countably many of them.  Thus, the argument relativizes.  In particular, the construction will produce a group with no escape in Turing degree $T$, but with a presentation in $T$.

	Since the presentation is in $T$, the set of escapes is $\Pi^0_1$ in $T$.  You can see this by observing that if an infinite sequence has a self-intersection, we will eventually know about it.  Trivially, there are escapes, so by the low basis theorem, there must be a low escape.  So, for any Turing degree, we have a group with no escapes in or below the given degree, but with an escape low relative to it.

\subsection{Inescapable Groups are Burnside Groups}\label{secBurnside}
	It is clear that in an inescapable group, all generators must have finite order, but must every group element have finite order?  An infinite group in which every element has finite order is called a Burnside group and the existence of Burnside groups was an open problem for some time before E. Golod constructed one in 1964~\cite{GolodShafarevich,Golod1964}.

	So our question can be rephrased as, must every inescapable group be a Burnside group?  It turns out the answer is yes, but it's not as obvious as it appears.  Suppose your group did have an element of infinite order.  The obvious thing to do would be to write this element of infinite order as a product of generators and simply repeat that sequence to produce an escape.  This certainly gives a computable sequence of generators that hits infinitely many elements of the group, but there is no reason this is a {\em non-self-intersecting} path.  Fortunately, we can find some other element of infinite order and an expression of it as a product of generators such that this naive construction does give a non-self-intersecting path.

	The general idea is to start with the obvious construction and cut out the loops.  Since our original element has infinite order, the constructed walk can only return to a given point a fixed, finite number of times.  So, we wait at each point of the walk until it returns to our current position for the last time, then we follow for one step, and repeat.  The only difficulty is knowing when the last return will be.  Fortunately, this is invariant under shifting by our infinite order element, so we can just record it in a finite table indexed by the generators in the expression of our infinite order element.

\begin{theorem}\label{thmInescBurn}
	Let $G=\langle g_0,\ldots,g_n\rangle$ be a group and $a\in G$ have infinite order.  Then there exists $b=\prod_{j=0}^{k-1}{g_{i_j}}\in G$ such that $\prod_{j=0}^N g_{i_{j\bmod{k}}}$ is distinct for each $N$.
\end{theorem}
\begin{proof}
	Let $a=\prod_{i=0}^{m-1} h_i$ with $h_i\in\{g_0,\ldots,g_n\}$ be an expression for $a$ of minimal length.  Define 
	\begin{equation*}
	\delta(r,s,M)=\left(\prod_{i=r+1}^{m-1}h_i\right)a^M\left(\prod_{j=0}^{s-1}h_j\right)
	\end{equation*}
	and consider relations of the form $\delta(r,s,M)=e$ with $M\geq 0$.  Fixing $r$ and $s$, there is at most one $M$ for which $\delta(r,s,M)=e$ since $a$ has infinite order.  Similarly, fixing $r$ and $M$, there is at most one $s$ for which $\delta(r,s,M)=e$ since we chose an expression for $a$ of minimal length.  This allows us to define the following functions, $\alpha:[0,m-1]\rightarrow\powerset(\N\times[0,m-1])$, $\beta:[0,m-1]\rightarrow\N$, and $\gamma:[0,m-1]\rightarrow[0,m-1]$ by
	\begin{eqnarray*}
		\alpha(r)&=&\{(M,s)|\delta(r,s,M)=e\}\\
		\beta(r)&=&\begin{cases}0\text{ if }\alpha(r)=\emptyset\\ 
			\max_{(M,s)\in\alpha(r)}M\text{ otherwise}\end{cases}\\
		\gamma(r)&=&\begin{cases}s\text{ if }(\beta(r),s)\in\alpha(r)\\
			(r+1)\bmod{m}\text{ if }\alpha(r)=\emptyset\end{cases}
	\end{eqnarray*}

	Notice that $\alpha(r)$ is always a finite set since $s$ can take on at most finitely many values and for each value, there is at most one $M$ such that $(M,s)\in\alpha(r)$.  Therefore, whenever $\alpha(r)$ is non-empty, $\beta(r)$ exists and is attained by exactly one element, $(\beta(r),s')$ of $\alpha(r)$.  By definition, $\gamma(r)=s'$, so $\gamma$ is well-defined.

	\begin{lemma}\label{prefixes}
		For all $n$, there is a $k_n$ such that 
		\begin{equation*}
			\prod_{i=1}^n h_{\gamma^{(i)}(m-1)}=a^{k_n}h_0\ldots h_{\gamma^{(n)}(m-1)}
		\end{equation*}
		where the sequence $\{k_n\}$ is defined by
		\begin{eqnarray*}
			k_0&=&-1\\
			k_{n+1}&=&k_n+\begin{cases}0\text{ if }\alpha(\gamma^{(n)}(m-1))=\emptyset\text{ and }\gamma^{(n)}(m-1)\neq m-1\\
				1\text{ if }\alpha(\gamma^{(n)}(m-1))=\emptyset\text{ and }\gamma^{(n)}(m-1)=m-1\\
				\beta(\gamma^{(n)}(m-1))+1\text{ otherwise}\end{cases}
		\end{eqnarray*}
	\end{lemma}
	\begin{proof}
		We proceed by induction.  If $n=0$, the product on the left side is empty, so we can take $k_0=-1$.

		For the induction step, suppose the lemma holds for $n-1$.  Then we can apply the induction hypothesis to reduce the problem to
		\begin{equation*}
			a^{k_{n-1}}h_0\ldots h_{\gamma^{(n-1)}(m-1)}h_{\gamma^{(n)}(m-1)}=a^{k_n}h_0\ldots h_{\gamma^{(n)}(m-1)}
		\end{equation*}
		or more simply,
		\begin{equation*}
			h_0\ldots h_{\gamma^{(n-1)}(m-1)}=a^{k_n-k_{n-1}}h_0\ldots h_{\gamma^{(n)}(m-1)-1}
		\end{equation*}
		If $\alpha(\gamma^{(n-1)}(m-1))$ is empty, then
		 \begin{equation*}
			\gamma^{(n)}(m-1)=(\gamma^{(n-1)}(m-1)+1)\bmod{m}
		\end{equation*}
		If $\gamma^{(n-1)}(m-1)=m-1$, then the product on the left is $a$ and the product on the right is $a^{k_n-k_{n-1}}$, so take $k_n=k_{n-1}+1$.  Otherwise, take $k_n=k_{n-1}$.

		On the other hand, if $\alpha(\gamma^{(n-1)}(m-1))$ is non-empty, $\gamma^{(n)}(m-1)=s$ for $(M,s)\in\alpha(\gamma^{(n-1)}(m-1))$ where $M=\beta(\gamma^{(n)}(m-1))$.  Since $(M,s)\in\alpha(\gamma^{(n-1)}(m-1))$,
		\begin{eqnarray*}
			h_0\ldots h_{\gamma^{(n-1)}(m-1)}&=&h_0\ldots h_{\gamma^{(n-1)}(m-1)}\delta(\gamma^{(n-1)}(m-1),s,M)\\
			&=&h_0\ldots h_{\gamma^{(n-1)}(m-1)}h_{\gamma^{(n-1)}(m-1)+1}\ldots h_{m-1}a^Mh_0\ldots h_{s-1}\\
			&=&a^{M+1}h_0\ldots h_{\gamma^{(n)}(m-1)-1}
		\end{eqnarray*}
		so we can take $k_n=M+1+k_{n-1}=\beta(\gamma^{(n)}(m-1))+1+k_{n-1}$.
	\end{proof}

	\begin{lemma}
		The sequence $h_{\gamma(m-1)},h_{\gamma(\gamma(m-1))},h_{\gamma^{(3)}(m-1)},\ldots$ corresponds to a non-self-intersecting path
	\end{lemma}
	\begin{proof}
		Suppose $\prod_{i=1}^{l_1}h_{\gamma^{(i)}(m-1)}=\prod_{i=1}^{l_2}h_{\gamma^{(i)}(m-1)}$ with $l_2>l_1$.  Then, by Lemma~\ref{prefixes}, 
		\begin{equation*}
			a^{k_{l_1}}h_0\ldots h_{\gamma^{(l_1)}(m-1)}=a^{k_{l_2}}h_0\ldots h_{\gamma^{(l_2)}(m-1)}
		\end{equation*}
		or,
		\begin{equation*}
			e=h_{\gamma^{(l_1)}(m-1)+1}\ldots h_{m-1}a^{k_{l_2}-k_{l_1}-1}h_0\ldots h_{\gamma^{(l_2)}(m-1)}
		\end{equation*}
		Then there are two cases.

		First, if $k_{l_1}=k_{l_2}$, then $\gamma^{(l_1)}(m-1)=\gamma^{(l_2)}(m-1)$ by the minimality of the expression for $a$.  In addition, for all $i$ with $l_1\leq i<l_2$, $\alpha(\gamma^{(i)}(m-1))=\emptyset$ and $\gamma^{(i)}(m-1)\neq m-1$.  Therefore, for all such $i$, \begin{equation*}
			\gamma^{(i+1)}(m-1)=\gamma^{(i)}(m-1)+1\bmod{m}
		\end{equation*}
		Since $\gamma^{(l_1)}(m-1)=\gamma^{(l_2)}(m-1)$, $l_2-l_1\geq m$.  But then for some $l_1\leq j<l_2$, $\gamma^{(j)}(m-1)=m-1$, which is a contradiction.

		Second, if $k_{l_1}>k_{l_2}$, then $(k_{l_2}-k_{l_1}-1,\gamma^{(l_2)}(m-1)+1)\in\alpha(\gamma^{(l_1)}(m-1))$.  Then $k_{l_1+1}\geq k_{l_1}+(k_{l_2}-k_{l-1}-1)+1=k_{l_2}$.  Since $l_2\geq l_1+1$ and the sequence of $k$'s is non-decreasing, $l_2=l_1+1$.  We can then calculate $\gamma^{(l_1+1)}(m-1)$ directly.  Note that
		\begin{eqnarray*}
			(\beta(\gamma^{(l_1)}(m-1)),\gamma^{(l_2)}(m-1)+1)&=&(k_{l_1+1}-k_{l_1}-1,\gamma^{(l_2)}(m-1)+1)\\
			&=&(k_{l_2}-k_{l_1}-1,\gamma^{(l_2)}(m-1)+1)\\
			&\in&\alpha(\gamma^{(l_1)}(m-1))
		\end{eqnarray*}
		So
		\begin{equation*}
			\gamma^{(l_1+1)}(m-1)=\gamma^{(l_2)}(m-1)+1=\gamma^{(l_1+1)}(m-1)+1
		\end{equation*}
		which is again a contradiction.
	\end{proof}

	Since $\gamma$ is a function from a finite set to itself, the sequence,
	\begin{equation*}
		\gamma(m-1),\gamma(\gamma(m-1)),\gamma^{(3)}(m-1)\ldots
	\end{equation*}
	is eventually periodic.  Let $\gamma^{(j_1)}(m-1),\ldots,\gamma^{(j_2)}(m-1)$ be a single period and take $b=\prod_{i=j_1}^{j_2}h_{\gamma^{(i)}(m-1)}$.  This expression of $b$ as a product of generators clearly fits the theorem.
\end{proof}

\begin{corollary}
	Inescapable groups are Burnside groups
\end{corollary}
\begin{proof}
	Suppose $(G,S)$ is a tape graph and $G$ is not a Burnside group.  Then $G$ has an element of infinite order.  By Theorem~\ref{thmInescBurn}, we can find some $i_0,\ldots,i_{k-1}$ such that
	\begin{equation*}
		\prod_{j=0}^N g_{i_{j\bmod{k}}}
	\end{equation*}
	is distinct for each $N$.  The sequence $\{g_{i_{j\bmod{k}}}\}_{j\in\N}$ is computable (in fact, it's regular), so $G$ is not inescapable.
\end{proof}

\subsection{Further Remarks}
	As was previously mentioned, generalizing to multiple tapes and multiple heads in this context is done in the same way as in the context of standard Turing machines.  For example, a one-tape machine with multiple heads can be simulated by a one-tape machine with one head.  The locations of all the heads can simply be marked on the tape with special symbols and for each step of the simulated machine, the simulating machine can search through the entire tape for all the head symbols and then update the tape accordingly.  Similarly, a machine with multiple, identical tapes can be simulated by a machine with a single tape of the same type.  Simply enlarge the tape alphabet to tuples of tape symbols together with special symbols for the head on each tape and follow the procedure for multiple heads.

	The only major difference is in a machine with multiple, different tapes.  In this case, the class of functions computable by the machine is the class of functions in or below the join of the r.e. degrees of the word problems of the tapes.  Just as in Section~\ref{secTDegrees}, this machine is mutually simulatable with a standard Turing machine with oracles for the word problem of each tape.  In fact, this case subsumes the multiple-head, single-tape and multiple-head, multiple-tape cases.

\bibliographystyle{amsplain}
\bibliography{../bibliography}
\end{document}